\documentclass[reqno, twoside]{article}
\input{preamble.sty}
\input{macros.sty}

\begin{document}
    \title{\textbf{\normalsize\MakeUppercase{The Ribes-Zalesskiĭ product theorem via foldings and extensions}}}
    \author{\footnotesize\MakeUppercase{Zhaoshen Zhai}}
    \date{}
    \maketitle
    \freefootnote{\textit{Date}: \today.}
    \freefootnote{2020 \textit{Mathematics Subject Classification.} 57M07, 20E10, 20E18.}
    \freefootnote{\textit{Key words and phrases.} Product separability, pseudovariety, profinite, pro-$\mathbf{V}$, foldings, immersions, $p$-elementary extensions.}

    \begin{center}
        \vspace{-0.3in}
        \begin{minipage}{0.85\textwidth}\setstretch{0.9}
            {\footnotesize{\textsc{Abstract.}} We present an exposition of the Auinger-Steinberg proof of the Ribes-Zalesskiĭ product theorem for pro-$\mathbf{V}$ topologies, where $\mathbf{V}$ is a pseudovariety of groups closed under extensions with abelian kernel. This proof is self-contained and is accessible to those acquainted with coverings of graphs, and as such, it provides an easy entry point to various other deep theorems for which the product theorem is formally equivalent to.}
        \end{minipage}
        \vspace{0.1in}
    \end{center}

    \section*{Introduction}

    \subsection*{The product theorem}

    The purpose of this note is to give a self-contained geometric proof of a generalization of the Ribes-Zalesskiĭ product theorem \cite{RZ93}, which in its original form states that the product $H_1\cdots H_n$ of finitely generated subgroups $H_1,\dots,H_n$ of a free group $F$ is closed in the profinite topology of $F$. This is a vast generalization of Hall's theorem \cite{Hal49}, which is the case when $n=1$, and has spurred much interaction between profinite group theory, semigroup theory, geometric group theory, and model theory.
    \begin{itemize}
        \item On the semigroup-theoretic side, this theorem was first conjectured by Pin and Reutenauer \cite{PR91}, who used it to establish Rhodes' type II conjecture, a major conjecture in semigroup theory at that time. The type II conjecture was also independently proven by Ash \cite{Ash91}, using completely different techniques (see \cite{HMPR91}), and since then, there has been major effort in unifying and reinterpreting the techniques of Ash and Ribes-Zalesskiĭ under a common framework; see, for instance, \cite{Ste02, AS03}.
            \vspace{-0.2in}
        \item The product theorem can also be thought of as a strong \textit{separability} property of free groups, generalizing Hall's Theorem that free groups are \textit{subgroup separable}, so groups satisfying the product theorem are said to be \textit{product separable}. Beyond the original product theorem on free groups, You \cite{You94} showed that surface groups are product separable. In another direction, Minasyan \cite{Min06} showed that the product of $n$ quasiconvex subgroups of geometrically-finite extended residually finite hyperbolic groups is closed in its profinite topology, which also generalizes the original product theorem.
            \vspace{-0.05in}
        \item Its connection with model theory appeared when Herwig and Lascar \cite{HL00} showed that the product theorem is formally equivalent to extension properties of partial automorphisms (\textit{EPPA}) of certain finite relational structures, which is important due to its connections with topological dynamics, Ramsey theory, and the study of automorphism groups of first-order structures; see, for instance \cite{The15, Las03}. This formal equivalence has also been generalized by Coulbois \cite{Cou01}, who in turn showed that groups satisfying the product theorem is closed under free products. The EPPA for metric spaces \cite{Sol05} has also been related to the product theorem by \cite[\S8]{Sab14}, where a new proof of the EPPA for metric space is given using the product theorem, and also by Rosendal \cite{Ros11} and Doucha-Malicki \cite{DM19}.
    \end{itemize}

    As it turns out, stronger forms of the product theorem were needed for applications to both semigroup theory and model theory. To state this strengthening, we need some more terminology. Let $\mathbf{V}$ be a \textit{pseudovariety} of groups, i.e., a class of finite groups closed under subgroups, quotients, and finite direct products. We say that $\mathbf{V}$ is \textit{closed under extensions} (\textit{with abelian kernel}) if $G\in\mathbf{V}$ whenever $N,G/N\in\mathbf{V}$ for any normal (abelian) subgroup $N\nsubeq G$. For a group $G$, the collection of normal subgroups $N\nsubeq G$ such that $G/N\in\mathbf{V}$ forms a neighborhood base around the identity, which generates the \textit{pro-$\mathbf{V}$ topology} of $G$.

    The main theorem of this note is as follows. It was first established by Ribes-Zalesskiĭ \cite{RZ94} when $\mathbf{V}$ is closed under all extensions, and by Auinger-Steinberg \cite{AS05} in the more general form stated below.

    \begin{mainTheorem}\label{mthm}
        Let $\mathbf{V}$ be a pseudovariety of groups that is closed under extensions with abelian kernel and let $F$ be a free group. If $H_1,\dots,H_n$ are finitely generated subgroups of $F$ which are closed in the pro-$\mathbf{V}$ topology of $F$, then their $n$-product coset $H_1\cdots H_n$ is also closed in the pro-$\mathbf{V}$ topology of $F$.
    \end{mainTheorem}

    In view of Hall's Theorem, this generalizes the product theorem. Our impetus for this generalization lies in our desire to study a certain finer profinite topology on $F$, the so-called \textit{pro-odd} topology, given by taking $\mathbf{V}$ to be the pseudovariety of all odd-order finite groups. Herwig and Lascar \cite{HL00} showed that the EPPA for the class of tournaments is equivalent to a characterization of the subgroups of $F_2$ which are closed in the pro-odd topology. Generalizing this, Huang, Pawliuk, Sabok, and Wise \cite{HPSW19} proved the equivalence of EPPA for classes of hypertournaments with characterizations of subgroups of the free groups closed in other profinite topologies. Both these results rely on Theorem \ref{mthm}; for details, see \cite{Che25}.

    The semigroup-theoretic side of Theorem \ref{mthm} stems from its connections with universal algebra \cite{Alm95} and profinite semigroup theory \cite{AW95}, where results on certain classes of pseudovarieties like Theorem \ref{mthm} can be used to shed light on the structure of finite semigroups.

    \vspace{-0.02in}
    \subsection*{Proofs of the product theorem}

    Originally, Ribes and Zalesskiĭ \cite{RZ93} proved the product theorem in the profinite case (where the hypothesis that the subgroups $H_1,\dots,H_n$ are closed can be dropped, by Hall's Theorem), but their proof generalizes verbatim to the pro-$\mathbf{V}$ case when $\mathbf{V}$ is closed under extensions \cite{RZ94}. There, a theory of \textit{profinite graphs}, developed by \cite{GR78, MZ89} (see also \cite{Rib17}), takes center stage, where the structure of the pseudovariety $\mathbf{V}$ imposes strong geometric constraints on the \textit{profinite Cayley graphs} of pro-$\mathbf{V}$ completions of free groups; namely, these profinite graphs are \textit{pro-$\mathbf{V}$ trees} if and only if $\mathbf{V}$ is closed under extensions with abelian kernel \cite{AW95}. Auinger \cite{Aui17}, and later also with Bors \cite{AB19}, showed that the product theorem applies to other much larger classes of groups, which are not necessarily pseudovarieties, but which are defined purely `geometrically' as in the characterization above.

    However, the machinery of profinite graphs is quite heavy and might not appeal to those interested only in its applications. Luckily, the original version of the product theorem (that is, in the profinite case) now has several proofs, stemming from either the formal equivalence between the theorems of Ash, Herwig-Lascar, and Ribes-Zalesskiĭ \cite{HL00, Cou01, Ste02, Ros11, ABO25}, or by other geometric means \cite{AS05, Min06}.

    The situation is less desirable in the general pro-$\mathbf{V}$ case of the product theorem. Although the aforementioned condition on profinite Cayley graphs of free groups holds when $\mathbf{V}$ is closed under extensions with abelian kernel, the original proof \cite{RZ93} requires the stronger condition that $\mathbf{V}$ is closed under \textit{all} extensions. Even forgoing this generalization, it is not clear whether any of the alternative proofs of the profinite case generalizes to the pro-$\mathbf{V}$ case, except for that of Auinger-Steinberg \cite{AS05}. It is thus the goal of this note to recast their proof in a more elementary geometric language, so as to make it as self-contained as possible.

    \vspace{-0.02in}
    \subsection*{Organization}

    The proof of Theorem \ref{mthm} we present here follows \cite{AS05}, but at a more leisurely pace. We start by giving some basic characterizations of the pro-$\mathbf{V}$ topology, which are more amenable to combinatorial and geometric tools. We then take a didactic approach by first proving the case $n=1$ (i.e., Hall's Theorem) using \textit{Stallings' foldings} \cite{Sta83} and elementary covering space theory of graphs, as these techniques generalize directly to the general case. These Stallings' foldings give rise to immersions of graphs, which can be extended to covering graphs; crucial to the pro-$\mathbf{V}$ case is that we need to control the group of deck transformations of these coverings, which we will elaborate on in due time.

    In the general case, we will need to augment these covering graphs with additional information depending on the pseudovariety $\mathbf{V}$, which we do so by the so-called \textit{universal $p$-elementary extension} of a group \cite{Els99}, for some prime $p$. We will explain how these extensions naturally arise in the proof.

    \vspace{-0.02in}
    \subsection*{Acknowledgements}

    I would like to thank Professor Marcin Sabok for supervising me for this project, for his consistent support, patience, and feedback, and for guiding me through this fulfilling research experience. I also thank Julian Cheng for helpful discussions and for explaining the original Ribes-Zalesskiĭ proof to me, and also I thank Karl Auinger for his valuable feedback on this note. This work was partially supported by McGill University’s SURA (Science Undergraduate Research Award) grant for the summer of 2025.

    \section{Preliminaries: Stallings' Foldings and Universal $p$-elementary Extensions}

    Although Theorem \ref{mthm} is about profinite topologies on free groups, surprisingly little profinite group theory is needed for its proof. All we need are the following three basic facts, which can also be found in \cite{RZ10}.

    \begin{lemma}\label{lem:open}
        Let $G$ be a group. A subgroup $H\subseq G$ is open in the pro-$\mathbf{V}$ topology of $G$ if and only if $H$ has finite index in $G$ and $G/H_G\in\mathbf{V}$, where $H_G\coloneqq\bigcap_{g\in G}gHg^{-1}$ is the {\em normal core} of $H$ in $G$.
    \end{lemma}
    \begin{proof}
        If $H\subseq G$ is open, then there is a normal subgroup $N\nsubeq G$ with $G/N\in\mathbf{V}$ such that $N\subseq H$, so $H$ has finite index in $G$. Moreover, since $N\subseq H_G$, we have $G/N\onto G/H_G$, and thus $G/H_G\in\mathbf{V}$ as $\mathbf{V}$ is closed under quotients. Conversely, if $G/H_G\in\mathbf{V}$, then $H_G$ is open, and hence so is $H$.
    \end{proof}

    \begin{lemma}\label{lem:closed}
        Let $G$ be a group. The following are equivalent for a subgroup $H\subseq G$.
        \begin{enumerate}
            \item $H$ is closed in the pro-$\mathbf{V}$ topology of $G$.
                \vspace{-0.05in}
            \item $H$ is the intersection of all subgroups of $G$ containing $H$ which are open in the pro-$\mathbf{V}$ topology of $G$.
                \vspace{-0.05in}
            \item For each $g\not\in H$, there is a finite group $K\in\mathbf{V}$ and a morphism $\phi:G\onto K$ such that $\phi(g)\not\in\phi(H)$.
        \end{enumerate}
    \end{lemma}
    \begin{proof}
        Let $\mc{N}$ denote all normal subgroups $N\nsubeq G$ such that $G/N\in\mathbf{V}$. Then (1) $\Rightarrow$ (2) since the intersection of all open subgroups of $G$ containing $H$ coincides with $\bigcap_{N\in\mc{N}}HN$, since if $x\in\bigcap_{N\in\mc{N}}HN$ and $H\subeq K\subseq_oG$, then $K_G\in\mc{N}$ by Lemma \ref{lem:open}, and so $x\in HK_G\subseq K$.

        For (2) $\Rightarrow$ (3), note that whenever $g\not\in H$, there is an open normal subgroup $N\in\mc{N}$ containing $H$ such that $g\not\in N$, so $K\coloneqq G/N$ and the canonical map $\phi:G\onto G/N$ works. Finally, for (3) $\Rightarrow$ (1), this can be reversed by using $gN$ to separate $g\not\in H$ from $H$, where $N\coloneqq\ker\phi$.
    \end{proof}

    We will use the above characterization of closed subgroups of free groups in the eventual proof of Theorem \ref{mthm}. In that context, we first note that since $H_1,\dots,H_n$ are finitely generated, they are enclosed in a finitely generated free factor of $F$. The following lemma allows us to assume, without loss of generality, that $F$ has finite rank in the first place.

    \begin{lemma}\label{lem:free_factor}
        If $H$ is a free factor of $G$, then the pro-$\mathbf{V}$ topology on $H$ coincides with the topology induced from the pro-$\mathbf{V}$ topology on $G$.
    \end{lemma}
    \begin{proof}
        The pro-$\mathbf{V}$ topology on $H$ is in general finer than the one induced from $G$, since, in the notation of the proof of Lemma \ref{lem:closed}, if $N\in\mc{N}$, then the natural map $H/(H\cap N)\to G/N\in\mathbf{V}$ is injective, so $H/(H\cap N)\in\mathbf{V}$.

        Conversely, note that if $G=H\ast K_0$, then $G=H\semiact K$ where $K$ is the normal closure of $K_0$ in $G$. Now take $N\nsubeq H$ such that $H/N\in\mathbf{V}$ and note that $G/KN\iso H/N\in\mathbf{V}$, so $KN$ is open in the pro-$\mathbf{V}$ topology of $G$. Observe that $N=H\cap KN$, so $N$ is open in the induced topology, as desired.
    \end{proof}

    \begin{notation*}
        In view of the Lemmas \ref{lem:closed}(3) and \ref{lem:free_factor}, we fix a free group $F$ generated by some fixed finite set $X$, and other than subgroups of $F$, we make the convention throughout the entire note that all groups are \textit{finite} and generated by $X$. For a \textit{word} $w\in(X\sqcup X^{-1})^\ast$ and a group $G$, this allows to define the \textit{value} of $w$ in $G$ as the image of $w$ under the canonical map $(X\sqcup X^{-1})^\ast\onto F\onto G$, denoted $[w]_G$.

        Graphs are defined following Serre \cite{Ser80}, except that we only consider finite graphs, and we label each edge $e$ by a letter $\ell(e)\in X\sqcup X^{-1}$; if $\ell(e)\in X$, then $e$ is said to be \textit{positively-oriented}. For an edge $e$, let $\alpha(e)$ and $\omega(e)$ denote its initial and terminal vertices, respectively, and let $\bar{e}$ be its reverse edge; the same notations apply to paths as well. If $G$ is a group, any path $\gamma$ in a graph induces an element $[\gamma]_G\coloneqq[\ell(\gamma)]_G\in G$, which gives rise to reduced paths in the Cayley graph $\mc{C}(G)$ with respect to $X$, starting at a given vertex.

        By the \textit{rose}, we mean the graph with a single vertex and $|X|$ loops, each labelled by a distinct element in $X$. A map $f$ of graphs is an \textit{immersion} (resp. \textit{covering}) if, for each vertex $v$ of its domain, $f$ restricts to an injection (resp. bijection) on the \textit{star} $\l\{e\st\alpha(e)=v\r\}$ of $v$. We say that a graph $\Gamma$ is an \textit{immersion}/\textit{covering} if the natural map from $\Gamma$ to the rose is such.
    \end{notation*}

    \subsection{The Stallings graph}

    Let $H$ be a finitely generated subgroup of $F$. We associate to $H$ its \textit{Stallings graph} $\mc{S}(H)$, which is essentially a connected graph with fundamental group $H$ that expands to a covering. To see why such a graph is useful, let us use it to give a short proof of Hall's Theorem.

    \begin{theorem}[\cite{Hal49}]
        Every finitely generated subgroup of a free group is closed in its profinite topology.
    \end{theorem}

    To show that $H$ is closed in the profinite topology of $F$, it suffices to find, for each word $w\not\in H$, a finite index normal subgroup $N\nsubeq F$ such that $Nw\cap H=\e$. Equivalently, it suffices to find a finite group $K$ such that $[w]_K\not\in[H]_K$. These two groups are related by $K\iso F/N$; see the proof of Lemma \ref{lem:closed}.

    It is known (see, for instance, \cite{Ser80}) that every finite index subgroup of $F$ is a free group that can be represented as the fundamental group of a covering. Thus, to construct $N$, we can start with $\mc{S}(H)$, attach to it a path $\gamma_w$ labelled $w$, and expand it to a cover $\mc{S}(H)^+$ in such a way that $\gamma_w$ does not collapse to a loop. The deck transformations of $\mc{S}(H)^+$ is a group $K\iso F/\pi_1(\mc{S}(H)^+)_F$, so $N\coloneqq\pi_1(\mc{S}(H)^+)_F$ is as desired.

    Let us now formalize these arguments and the construction of $\mc{S}(H)$; see Figures \ref{fig:fold} and \ref{fig:example_stallings} for reference.

    \begin{figure}
        \center
        \begin{tikzpicture}
            \fill (0,0) circle (0.02in);
            \fill (2,0.5) circle (0.02in);
            \fill (2,-0.5) circle (0.02in);

            \draw[->-=0.52] (0,0) -- node[above] {\small$e_1$} (2, 0.5);
            \draw[->-=0.52] (0,0) -- node[below] {\small$e_2$} (2,-0.5);

            \fill (2.5,0) circle (0) node {\small or};

            \fill (5,0) circle (0.02in);
            \fill (3,0.5) circle (0.02in);
            \fill (3,-0.5) circle (0.02in);

            \draw[->-=0.52] (3, 0.5) -- node[above] {\small$e_1$} (5,0);
            \draw[->-=0.52] (3,-0.5) -- node[below] {\small$e_2$} (5,0);

            \draw[->, decorate, decoration = {zigzag, segment length=4, amplitude=.9,post=lineto, post length=2pt}] (5.74,0) -- node[above] {\small Fold} (7.25,0);

            \begin{scope}[xshift=8cm]
                \fill (0,0) circle (0.02in);
                \fill (2,0) circle (0.02in);
                \draw[->-=0.53] (0,0) -- node[above] {\small$e_1\sim e_2$} (2,0);
            \end{scope}
        \end{tikzpicture}           
        \begin{minipage}{0.95\textwidth}
            \vspace{0.1in}
            \caption{Folding an admissible pair $(e_1,e_2)$ common initial vertices. In the second case, observe that $(\bar{e_1},\bar{e_2})$ is an admissible pair, so we can also fold their initial vertices together.}
            \label{fig:fold}
            \vspace{-0.1in}
        \end{minipage}
    \end{figure}
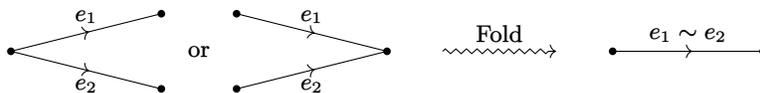

    \begin{figure}
        \center
        \begin{tikzpicture}
            \def\loopAngle{10}
            \def\loopTightness{0.1}

            \draw[->-=0.52, blue] (0,0) -- node[below] {\footnotesize$y$} (1,-0.5);
            \draw[->-=0.52, blue] (0,0) -- node[below] {\footnotesize$y$} (-1,-0.5);

            \fill (0,0) circle (0.02in) node[below] {\small$v$};
            \fill (2,0) circle (0.02in);
            \fill ( 1, 0.5) circle (0.02in);
            \fill ( 1,-0.5) circle (0.02in);
            \fill (-1, 0.5) circle (0.02in);
            \fill (-1,-0.5) circle (0.02in);

            \draw[->-=0.52] (0,0) -- node[above] {\footnotesize$x$} (1,0.5);
            \draw[->-=0.52] (1,0.5) -- node[above] {\footnotesize$y$} (2,0);
            \draw[->-=0.52] (1,-0.5) -- node[below] {\footnotesize$x$} (2,0);
            \draw[->-=0.52] (0,0) -- node[above] {\footnotesize$y$} (-1,0.5);
            \draw[->-=0.52] (-1,-0.5) -- node[left] {\footnotesize$x$} (-1,0.5);

            \draw[->, decorate, decoration = {zigzag, segment length=4, amplitude=.9,post=lineto, post length=2pt}] (2.75,0) -- node[above] {\small Fold \#1} (4.25,0);

            \begin{scope}[xshift=5cm]
                \draw[->-=0.52, blue] (1,-0.5) -- node[below] {\footnotesize$x$} (2,0);
                \draw[->-=0.55, blue] (1,-0.5) -- node[below] {\footnotesize$x$} (0,-1);

                \fill (0,0) circle (0.02in) node[left]{\small$v$};
                \fill (2,0) circle (0.02in);
                \fill (1, 0.5) circle (0.02in);
                \fill (1,-0.5) circle (0.02in);
                \fill (0,-1) circle (0.02in);

                \draw[->-=0.52] (0,0) -- node[above] {\footnotesize$x$} (1,0.5);
                \draw[->-=0.52, thick] (0,0) -- node[below] {\footnotesize$y$} (1,-0.5);
                \draw[->-=0.52] (1, 0.5) -- node[above] {\footnotesize$y$} (2,0);
                \draw[->-=0.6] (0,0) -- node[left] {\footnotesize$y$} (0,-1);

                \draw[->, decorate, decoration = {zigzag, segment length=4, amplitude=.9,post=lineto, post length=2pt}] (2.75,0) -- node[above] {\small Fold \#2} (4.25,0);

                \begin{scope}[xshift=5cm]
                    \draw[->-=0.52, blue] (0,0) -- node[below] {\footnotesize$y$} (1,-0.5);
                    \draw[->-=0.52, blue] (0,0) -- node[above] {\footnotesize$y$} (2,0);

                    \fill (0,0) circle (0.02in) node[left] {\small$v$};
                    \fill (2,0) circle (0.02in);
                    \fill (1, 0.5) circle (0.02in);
                    \fill (1,-0.5) circle (0.02in);

                    \draw[->-=0.52] (0,0) -- node[above] {\footnotesize$x$} (1,0.5);
                    \draw[->-=0.52] (1, 0.5) -- node[above] {\footnotesize$y$} (2,0);
                    \draw[->-=0.52, thick] (1,-0.5) -- node[below] {\footnotesize$x$} (2,0);
                \end{scope}
            \end{scope}

            \coordinate (s) at (11,-0.75);
            \coordinate (e) at (0,-1.75);

            \draw[->, decorate, decoration = {zigzag, segment length=5, amplitude=.9,post=lineto, post length=2pt}] (s.south) .. controls (11,-3) and (0,0) .. (e.north);
            \node[below, rotate=-3] at (5.5,-1.45) {\small Fold \#3};

            \begin{scope}[xshift=-1cm, yshift=-2.5cm]
                \draw[->-=0.52, blue, thick] (0,0) -- node[above] {\footnotesize$y$} (2,0);
                \draw[->-=0.52, blue] (1, 0.5) -- node[above] {\footnotesize$y$} (2,0);

                \fill (0,0) circle (0.02in) node[left] {\small$v$};
                \fill (2,0) circle (0.02in);
                \fill (1, 0.5) circle (0.02in);

                \draw[->-=0.52] (0,0) -- node[above] {\footnotesize$x$} (1,0.5);

                \draw (2,0) to[out=-\loopAngle, in=-90] (2.75,-\loopTightness);
                \draw[->-=0.8] (2.75,-\loopTightness) -- node[right] {\footnotesize$x$} (2.75,\loopTightness);
                \draw (2.75,\loopTightness) to[out=90, in=\loopAngle] (2,0);

                \draw[->, decorate, decoration = {zigzag, segment length=4, amplitude=.9,post=lineto, post length=2pt}] (3.75,0) -- node[above]{\small Fold \#4} (5.25,0);

                \begin{scope}[xshift=6cm]
                    \fill (1,0) circle (0.02in) node[below] {\small$v$};
                    \fill (2,0) circle (0.02in);

                    \draw[->-=0.55, thick] (1,0) -- node[above] {\footnotesize$y$} (2,0);

                    \draw (2,0) to[out=-\loopAngle, in=-90] (2.75,-\loopTightness);
                    \draw[->-=0.8] (2.75,-\loopTightness) -- node[right] {\footnotesize$x$} (2.75,\loopTightness);
                    \draw (2.75,\loopTightness) to[out=90, in=\loopAngle] (2,0);

                    \draw (1,0) to[out=180+\loopAngle, in=-90] (0.25,-\loopTightness);
                    \draw[->-=0.8] (0.25,-\loopTightness) -- node[left] {\footnotesize$x$} (0.25,\loopTightness);
                    \draw (0.25,\loopTightness) to[out=90, in=180-\loopAngle] (1,0);
                \end{scope}
            \end{scope}
        \end{tikzpicture}
        \begin{minipage}{0.95\textwidth}
            \vspace{0.1in}
            \caption{The construction of $\mc{S}(H)$ for $H\coloneqq\langle xyx^{-1}y^{-1},yxy^{-1}\rangle$: starting from the top-left corner, the blue edges are the admissible edges to be folded in the next step, which results in a thicker edge. Note that the last fold requires us to identify edges with a common terminal vertex instead.}
            \label{fig:example_stallings}
            \vspace{-0.2in}
        \end{minipage}
    \end{figure}
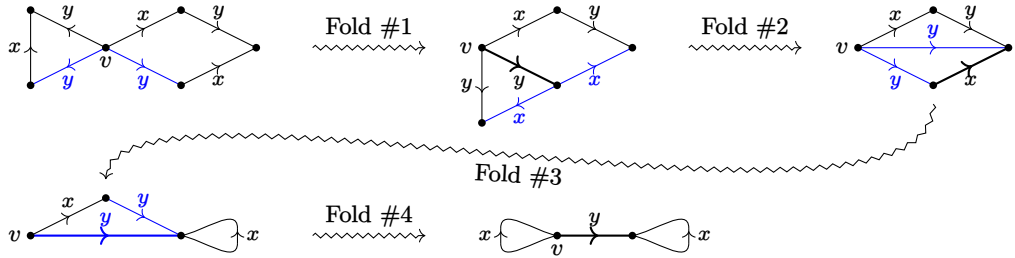

    \begin{definition}\label{def:fold}
        Let $\Gamma$ be a graph. A pair $(e_1,e_2)$ of edges in $\Gamma$ is said to be \textit{admissible} if $e_2\neq\bar{e}_1$, $\alpha(e_1)=\alpha(e_2)$, and $\ell(e_1)=\ell(e_2)$. For any such pair, we can \textit{fold} $\Gamma$ by identifying $e_1\sim e_2$, $\bar{e_1}\sim\bar{e_2}$, and $\omega(e_1)\sim\omega(e_2)$.

        The \textit{Stallings graph} of a finitely generated subgroup $H\subseq F$ is the graph $(\mc{S}(H),v)$, constructed as follows.
        \begin{enumerate}
            \item Write $H=\gen{w_1,\dots,w_m}$ for reduced words $w_i$, and consider the wedge of $m$ disjoint cycles $C_1,\dots,C_m$ with basepoint $v$, such that the length of $C_i$ is $|w_i|$. Starting from $v$, we label each $C_i$ by the word $w_i$.
                \vspace{-0.05in}
            \item Denote the resulting graph by $\Gamma$, which we then repeatedly fold until there are no admissible pairs of edges; since $\Gamma$ is finite, this process terminates. This yields the graph $\mc{S}(H)$.
        \end{enumerate}
    \end{definition}

    Since we have folded all admissible pairs of edges, there is an natural map from $\mc{S}(H)$ to the rose, which is injective on stars. Thus, $\mc{S}(H)$ is an immersion. The next two lemmas show that $\pi_1(\mc{S}(H),v)\iso H$ and that $\mc{S}(H)$ can be expanded to a cover. Since we shall work with other immersions too, we state them abstractly.

    \begin{lemma}\label{lem:folding}
        Suppose that $\Gamma$ is a graph in which every admissible pair of edges have distinct terminal vertices. Let $\Gamma'$ be the graph obtained by folding all admissible pairs in $\Gamma$. Then $\pi_1\Gamma\iso\pi_1\Gamma'$ and $\Gamma'$ is an immersion.
    \end{lemma}
    \begin{proof}
        Our hypothesis ensures that folding an admissible pair induces a homotopy equivalence between the geometric realizations of the two graphs, so $\pi_1\Gamma\iso\pi_1\Gamma'$; see Figure \ref{fig:fold}. That $\Gamma'$ is an immersion is clear.
    \end{proof}

    \begin{lemma}\label{lem:expand}
        Every immersion expands to a covering without adding new vertices.
    \end{lemma}
    \begin{proof}
        Let $\Gamma$ be an immersion with vertex set $V$. For each $x\in X$, there is a partial function $f_x:V\parto V$ sending $v\mapsto vx$, where $vx$ is the terminal vertex of an edge in the star of $v$ labelled by $x$, if it exists. Since $\Gamma$ is an immersion, these partial functions are all injective, and hence can all be extended to permutations of $V$. Adding edges to $\Gamma$ in accordance to these permutations then furnishes a cover, which is as desired.
    \end{proof}

    \begin{remark}
        For an immersion $\Gamma$, let $\Gamma^+$ denote any covering expansion of $\Gamma$ obtained as in the above lemma; since no vertices are added, the index of $\pi_1(\mc{S}(H)^+)$ in $F$ is the same as that of $H$. There are usually many such expansions (see Figure \ref{fig:example_expand}), and we will sometimes need to choose a specific one (see Lemma \ref{lem:deck}).
    \end{remark}

    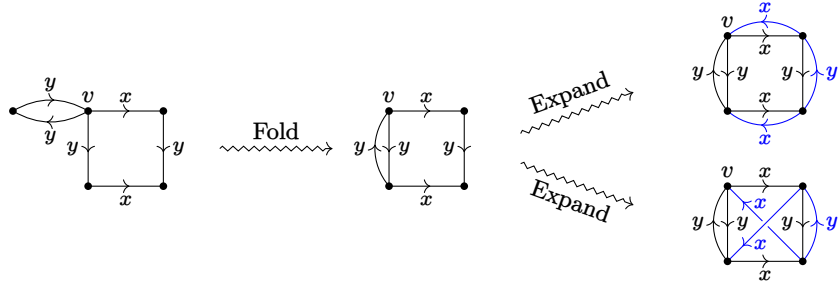
\begin{figure}
        \center
        \begin{center}
            \begin{tikzpicture}
                \fill (0,0) circle (0.02in) node[above] {\small$v$};
                \fill (1,0) circle (0.02in);
                \fill (1,-1) circle (0.02in);
                \fill (0,-1) circle (0.02in);
                \fill (-1,0) circle (0.02in);

                \draw[->-=0.55] (0,0) -- node[above] {\footnotesize$x$} (1,0);
                \draw[->-=0.55] (1,0) -- node[right] {\footnotesize$y$} (1,-1);
                \draw[->-=0.55] (0,-1) -- node[below] {\footnotesize$x$} (1,-1);
                \draw[->-=0.55] (0,0) -- node[left] {\footnotesize$y$} (0,-1);
                \draw[->-=0.55] (0,0) to[bend left] node[below] {\footnotesize$y$} (-1,0);
                \draw[->-=0.55] (-1,0) to[bend left] node[above] {\footnotesize$y$} (0,0);

                \draw[->, decorate, decoration = {zigzag, segment length=4, amplitude=.9,post=lineto, post length=2pt}] (1.75,-0.5) -- node[above]{\small Fold} (3.25,-0.5);

                \begin{scope}[xshift=4cm]
                    \fill (0,0) circle (0.02in) node[above] {\small$v$};
                    \fill (1,0) circle (0.02in);
                    \fill (1,-1) circle (0.02in);
                    \fill (0,-1) circle (0.02in);

                    \draw[->-=0.55] (0,0) -- node[above] {\footnotesize$x$} (1,0);
                    \draw[->-=0.55] (1,0) -- node[left] {\footnotesize$y$} (1,-1);
                    \draw[->-=0.55] (0,-1) -- node[below] {\footnotesize$x$} (1,-1);
                    \draw[->-=0.55] (0,0) to node[right] {\footnotesize$y$} (0,-1);
                    \draw[->-=0.55] (0,-1) to[bend left=40] node[left] {\footnotesize$y$} (0,0);

                    \draw[->, decorate, decoration = {zigzag, segment length=4, amplitude=.9,post=lineto, post length=2pt}] (1.75,-0.5+0.2) -- node[above, rotate=20]{\small Expand} (3.25,-0.5+0.75);
                    \draw[->, decorate, decoration = {zigzag, segment length=4, amplitude=.9,post=lineto, post length=2pt}] (1.75,-0.5-0.2) -- node[below, rotate=-20]{\small Expand} (3.25,-0.5-0.75);

                    \begin{scope}[xshift=4.5cm, yshift=1cm]
                        \draw[->-=0.55, blue] (1,0) to[bend right=40] node[above] {\footnotesize$x$} (0,0);
                        \draw[->-=0.55, blue] (1,-1) to[bend right=40] node[right] {\footnotesize$y$} (1,0);
                        \draw[->-=0.55, blue] (1,-1) to[bend left=40] node[below] {\footnotesize$x$} (0,-1);

                        \fill (0,0) circle (0.02in) node[above] {\small$v$};
                        \fill (1,0) circle (0.02in);
                        \fill (1,-1) circle (0.02in);
                        \fill (0,-1) circle (0.02in);

                        \draw[->-=0.55] (0,0) -- node[below] {\footnotesize$x$} (1,0);
                        \draw[->-=0.55] (1,0) -- node[left] {\footnotesize$y$} (1,-1);
                        \draw[->-=0.55] (0,-1) -- node[above] {\footnotesize$x$} (1,-1);
                        \draw[->-=0.55] (0,0) to node[right] {\footnotesize$y$} (0,-1);

                        \draw[->-=0.55] (0,-1) to[bend left=40] node[left] {\footnotesize$y$} (0,0);
                    \end{scope}

                    \begin{scope}[xshift=4.5cm, yshift=-1cm]
                        \draw[->-=0.55, blue] (1,-1) to[bend right=40] node[right] {\footnotesize$y$} (1,0);

                        \draw[->-=0.8, blue] (1,-1) -- node[xshift=-0.06cm, yshift=0.27cm] {\footnotesize$x$} (0,0);
                        \fill[white] (0.5,-0.5) circle (0.02in);
                        \draw[->-=0.8, blue] (1,0) -- node[xshift=-0.06cm, yshift=-0.27cm] {\footnotesize$x$} (0,-1);

                        \fill (0,0) circle (0.02in) node[above] {\small$v$};
                        \fill (1,0) circle (0.02in);
                        \fill (1,-1) circle (0.02in);
                        \fill (0,-1) circle (0.02in);

                        \draw[->-=0.55] (0,0) -- node[above] {\footnotesize$x$} (1,0);
                        \draw[->-=0.55] (1,0) -- node[left] {\footnotesize$y$} (1,-1);
                        \draw[->-=0.55] (0,-1) -- node[below] {\footnotesize$x$} (1,-1);
                        \draw[->-=0.55] (0,0) to node[right] {\footnotesize$y$} (0,-1);

                        \draw[->-=0.55] (0,-1) to[bend left=40] node[left] {\footnotesize$y$} (0,0);
                    \end{scope}
                \end{scope}
            \end{tikzpicture}
            \begin{minipage}{0.95\textwidth}
                \vspace{0.1in}
                \caption{The Stallings graph $\mc{S}(\langle xyx^{-1}y^{-1},y^2\rangle)$ admit two different extensions to covers of the rose.}
                \label{fig:example_expand}
                \vspace{-0.1in}
            \end{minipage}
        \end{center}
    \end{figure}

    Using these lemmas, we can now tie up the loose ends of the above sketch of Hall's Theorem: let $\mc{S}(H)_w$ be the graph obtained by attaching to $\mc{S}(H)$ a path $\gamma_w$ labelled $w$, where $\omega(\gamma_w)=v$ and $\alpha(\gamma_w)\not\in\mc{S}(H)$, which we then fold; see Figure \ref{fig:example_attach}. The fundamental group of $\mc{S}(H)_w$ remains $H$, and as $\mc{S}(H)_w$ is an immersion, we can expand it to a cover $\mc{S}(H)^+_w$ of the rose. Since $w\not\in H$, the path $\gamma_w$ is not a loop in $\mc{S}(H)_w$, so it is also not a loop in $\mc{S}(H)^+_w$. Thus the deck transformation of $\mc{S}(H)_w^+$ induced by $w^{-1}$ translates $v=\omega(\gamma_w)$ to $\alpha(\gamma_w)\neq v$, and this is not the case for those induced by any $h\in H$, so $w^{-1}$, and hence $w$, is separated from $H$ in the quotient. This ends the proof of Hall's Theorem.

    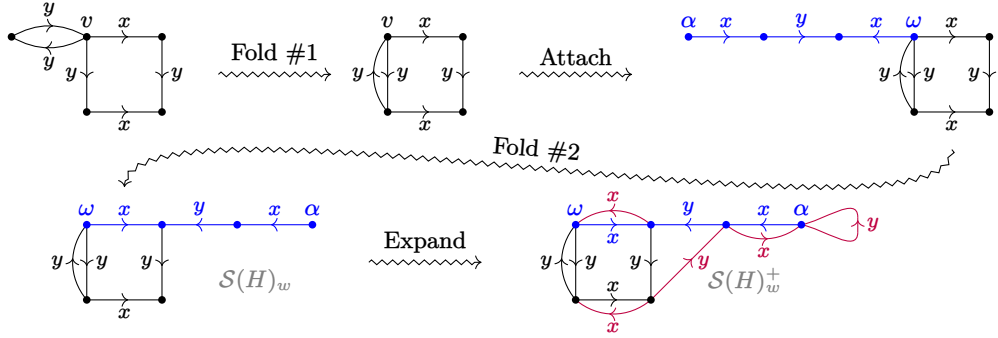
\begin{figure}
        \center
        \begin{center}
            \begin{tikzpicture}
                \def\loopAngle{10}
                \def\loopTightness{0.1}

                \fill (0,0) circle (0.02in) node[above] {\small$v$};
                \fill (1,0) circle (0.02in);
                \fill (1,-1) circle (0.02in);
                \fill (0,-1) circle (0.02in);
                \fill (-1,0) circle (0.02in);

                \draw[->-=0.55] (0,0) -- node[above] {\footnotesize$x$} (1,0);
                \draw[->-=0.55] (1,0) -- node[right] {\footnotesize$y$} (1,-1);
                \draw[->-=0.55] (0,-1) -- node[below] {\footnotesize$x$} (1,-1);
                \draw[->-=0.55] (0,0) -- node[left] {\footnotesize$y$} (0,-1);
                \draw[->-=0.55] (0,0) to[bend left] node[below] {\footnotesize$y$} (-1,0);
                \draw[->-=0.55] (-1,0) to[bend left] node[above] {\footnotesize$y$} (0,0);

                \draw[->, decorate, decoration = {zigzag, segment length=4, amplitude=.9,post=lineto, post length=2pt}] (1.75,-0.5) -- node[above]{\small Fold \#1} (3.25,-0.5);

                \begin{scope}[xshift=4cm]
                    \fill (0,0) circle (0.02in) node[above] {\small$v$};
                    \fill (1,0) circle (0.02in);
                    \fill (1,-1) circle (0.02in);
                    \fill (0,-1) circle (0.02in);

                    \draw[->-=0.55] (0,0) -- node[above] {\footnotesize$x$} (1,0);
                    \draw[->-=0.55] (1,0) -- node[left] {\footnotesize$y$} (1,-1);
                    \draw[->-=0.55] (0,-1) -- node[below] {\footnotesize$x$} (1,-1);
                    \draw[->-=0.55] (0,0) to node[right] {\footnotesize$y$} (0,-1);
                    \draw[->-=0.55] (0,-1) to[bend left=40] node[left] {\footnotesize$y$} (0,0);

                    \draw[->, decorate, decoration = {zigzag, segment length=4, amplitude=.9,post=lineto, post length=2pt}] (1.75,-0.5) -- node[above] {\small Attach} (3.25,-0.5);

                    \begin{scope}[xshift=7cm]
                        \fill (1,0) circle (0.02in);
                        \fill (1,-1) circle (0.02in);
                        \fill (0,-1) circle (0.02in);

                        \draw[->-=0.55] (0,0) -- node[above] {\footnotesize$x$} (1,0);
                        \draw[->-=0.55] (1,0) -- node[left] {\footnotesize$y$} (1,-1);
                        \draw[->-=0.55] (0,-1) -- node[below] {\footnotesize$x$} (1,-1);
                        \draw[->-=0.55] (0,0) to node[right] {\footnotesize$y$} (0,-1);
                        \draw[->-=0.55] (0,-1) to[bend left=40] node[left] {\footnotesize$y$} (0,0);

                        \draw[->-=0.55, blue] (-3,0) -- node[above] {\footnotesize$x$} (-2,0);
                        \draw[->-=0.55, blue] (-2,0) -- node[above] {\footnotesize$y$} (-1,0);
                        \draw[->-=0.55, blue] (0,0) -- node[above] {\footnotesize$x$} (-1,0);

                        \fill[blue] (-3,0) circle (0.02in) node[above] {\small$\alpha$};
                        \fill[blue] (-2,0) circle (0.02in);
                        \fill[blue] (-1,0) circle (0.02in);
                        \fill[blue] (0,0) circle (0.02in) node[above] {\small$\omega$};
                    \end{scope}
                \end{scope}

                \coordinate (s) at (11.5,-1.5);
                \coordinate (e) at (0.5,-2);

                \draw[->, decorate, decoration = {zigzag, segment length=4.5, amplitude=.9,post=lineto, post length=2pt}] (s.south) .. controls (11.5,-3) and (0.5,-0.5) .. (e.north);
                \node[below, rotate=-5] at (6,-1.25) {\small Fold \#2};

                \begin{scope}[yshift=-2.5cm]
                    \node[gray] at (2.25,-0.75) {\small$\mc{S}(H)_w$};
                    \fill (1,-1) circle (0.02in);
                    \fill (0,-1) circle (0.02in);

                    \draw[->-=0.55] (1,0) -- node[left] {\footnotesize$y$} (1,-1);
                    \draw[->-=0.55] (0,-1) -- node[below] {\footnotesize$x$} (1,-1);
                    \draw[->-=0.55] (0,0) to node[right] {\footnotesize$y$} (0,-1);
                    \draw[->-=0.55] (0,-1) to[bend left=40] node[left] {\footnotesize$y$} (0,0);

                    \draw[->-=0.55, blue] (0,0) -- node[above] {\footnotesize$x$} (1,0);
                    \draw[->-=0.55, blue] (3,0) -- node[above] {\footnotesize$x$} (2,0);
                    \draw[->-=0.55, blue] (2,0) -- node[above] {\footnotesize$y$} (1,0);

                    \fill[blue] (3,0) circle (0.02in) node[above] {\small$\alpha$};
                    \fill[blue] (2,0) circle (0.02in);
                    \fill[blue] (1,0) circle (0.02in);
                    \fill[blue] (0,0) circle (0.02in) node[above] {\small$\omega$};

                    \draw[->, decorate, decoration = {zigzag, segment length=4, amplitude=.9,post=lineto, post length=2pt}] (3.75,-0.5) -- node[above] {\small Expand} (5.25,-0.5);

                    \begin{scope}[xshift=6.5cm]
                        \draw[->-=0.55, purple] (2,0) to[bend right=40] node[below] {\footnotesize$x$} (3,0);
                        \draw[->-=0.55, purple] (1,-1) -- node[right]{ \footnotesize$y$} (2,0);
                        \draw[->-=0.55, purple] (1,0) to[bend right=40] node[above] {\footnotesize$x$} (0,0);
                        \draw[->-=0.55, purple] (1,-1) to[bend left=40] node[below] {\footnotesize$x$} (0,-1);

                        \draw[purple] (3,0) to[out=-\loopAngle, in=-90] (3.75,-\loopTightness);
                        \draw[->-=0.8, purple] (3.75,-\loopTightness) -- node[right] {\footnotesize$y$} (3.75,\loopTightness);
                        \draw[purple] (3.75,\loopTightness) to[out=90, in=\loopAngle] (3,0);

                        \node[gray] at (2.25,-0.75) {\small$\mc{S}(H)_w^+$};
                        \fill (1,-1) circle (0.02in);
                        \fill (0,-1) circle (0.02in);

                        \draw[->-=0.55] (1,0) -- node[left] {\footnotesize$y$} (1,-1);
                        \draw[->-=0.55] (0,-1) -- node[above] {\footnotesize$x$} (1,-1);
                        \draw[->-=0.55] (0,0) to node[right] {\footnotesize$y$} (0,-1);
                        \draw[->-=0.55] (0,-1) to[bend left=40] node[left] {\footnotesize$y$} (0,0);

                        \draw[->-=0.55, blue] (0,0) -- node[below] {\footnotesize$x$} (1,0);
                        \draw[->-=0.55, blue] (3,0) -- node[above] {\footnotesize$x$} (2,0);
                        \draw[->-=0.55, blue] (2,0) -- node[above] {\footnotesize$y$} (1,0);

                        \fill[blue] (3,0) circle (0.02in) node[above] {\small$\alpha$};
                        \fill[blue] (2,0) circle (0.02in);
                        \fill[blue] (1,0) circle (0.02in);
                        \fill[blue] (0,0) circle (0.02in) node[above] {\small$\omega$};
                    \end{scope}
                \end{scope}
            \end{tikzpicture}
            \begin{minipage}{0.95\textwidth}
                \vspace{0.1in}
                \caption{For $H\coloneqq\langle xyx^{-1}y^{-1},y^2\rangle$ and the word $w\coloneqq xyx^{-1}$, we attach to $v$ the terminal point of the path $\gamma_w$ (in blue), which we then fold to obtain the immersion $\mc{S}(H)_w$.}
                \label{fig:example_attach}
                \vspace{-0.2in}
            \end{minipage}
        \end{center}
    \end{figure}

    \begin{remark}\label{rem:strategy}
        The pseudovariety $\mathbf{V}$ disappeared from our discussion, and for good reason: it is \textit{not} generally the case that a finitely generated subgroup $H$ is closed in the pro-$\mathbf{V}$ topology of $F$. For this, we need to find a group $K\in\mathbf{V}$ that separates $w$ from $H$ in the quotient, but the above construction of $K$, as the group of deck transformations of a cover, is not enough to control the structure of $K$.

        In the context of Theorem \ref{mthm}, to prove that $H_1\cdots H_n$ is closed in the pro-$\mathbf{V}$ topology of a free group, we will \textit{extend} the group $G$ of deck transformations of the covering expansions of the Stallings graphs associated to $H_1,\dots,H_n$ and $w$. These extensions must be generic enough to lie in $\mathbf{V}$, yet also strong enough to separate $w$ from $H_1\cdots H_n$. Very roughly speaking, we will use an (iterated) extension of $G$ to force the Cayley graph of $G$ to be more `tree-like', and we will do so in such a way that the endpoints of certain paths stay put.
    \end{remark}

    \subsection{Extensions by abelian $p$-groups}

    Fix\vspace{-0.023in} a prime $p$ and let $E^+$ be the set of positively-oriented edges in $\mc{C}(G)$. Let $C_pE^+\coloneqq(\Z/p\Z)^{E^+}$ be the free $\Z/p\Z$-module on $E^+$, which admits a natural action by $G$. This gives rise to a semidirect product $C_pE^+\semiacted G$, and we define the \textit{universal $p$-elementary extension}\footnote{This terminology is justified by the following universal property (although we shall not need it here): for any extension $G'$ of $G$ by an elementary abelian $p$-group, the morphism $G^{\mathbf{Ab}_p}\onto G$ factors through $G'$ \cite{Els99}.} of $G$ as
    \vspace{-0.05in}
    \begin{equation*}
        G^{\mathbf{Ab}(p)}\coloneqq\l\langle(e_x,x)\st x\in X\r\rangle\subseq C_pE^+\semiacted G,
        \vspace{-0.05in}
    \end{equation*}
    where $e_x\in E^+$ is the edge in $\mc{C}(G)$ from $1$ to $x$. Given a word $w$, $G^{\mathbf{Ab}(p)}$ not only computes its image $[w]_G$ in $G$, by virtue of $G^{\mathbf{Ab}(p)}\onto G$, but also `keeps track' of the edges that $w$ traces out in $\mc{C}(G)$ as a path from $1$ to $[w]_G$, with multiplicity mod $p$. Formally, let $w(e)$ be the number of signed traversals of $e\in E^+$ by the path $\gamma_w:1\to[w]_G$ traced out by $w$ in $\mc{C}(G)$, and let $[w(e)]_p\coloneqq w(e)$ mod $p$; an induction on $|w|$ shows that
    \vspace{-0.05in}
    \begin{equation}\tag{$\ast$}\label{eqn:magic}
        [w]_{G^{\mathbf{Ab}(p)}}=\l(\sum\nolimits_{e\in E^+}[w(e)]_pe,[w]_G\r).
        \vspace{-0.05in}
    \end{equation}
    We will show that the additional information that $G^{\mathbf{Ab}(p)}$ carries is enough to implement the strategy hinted in Remark \ref{rem:strategy} for the case $n=2$; the general case requires an iterated extension.

    Before doing so, let us first show that the relevant groups lie in $\mathbf{V}$. This is the only time where we invoke the hypotheses that $\mathbf{V}$ is closed under extensions with abelian kernel and that the subgroups $H_1,\dots,H_n\subseq F$ are closed in the pro-$\mathbf{V}$ topology of $F$. Of course, these considerations are irrelevant in the profinite case.

    \begin{lemma}\label{lem:extension}
        For any group $G\in\mathbf{V}$, there is a prime $p$ such that $G^{\mathbf{Ab}(p)}\in\mathbf{V}$.
    \end{lemma}
    \begin{proof}
        First note that every non-trivial group in $\mathbf{V}$ contains a subgroup $\Z/p\Z$ for any prime $p$ dividing its order, so $\Z/p\Z\in\mathbf{V}$ for some prime $p$. Thus, in the above notation, $C_pE^+$ is an abelian $p$-group in $\mathbf{V}$. Since $\mathbf{V}$ is closed under extensions with abelian kernel and $G\in\mathbf{V}$, its extension $G^{\mathbf{Ab}(p)}$ lies in $\mathbf{V}$ as well.
    \end{proof}

    \begin{lemma}\label{lem:deck}
        Let $H$ be a finitely generated subgroup of $F$. If $H$ is closed in the pro-$\mathbf{V}$ topology of $F$, then its Stallings graph $\mc{S}(H)$ extends to some cover $\mc{S}(H)^+$ whose group of deck transformations lie in $\mathbf{V}$.
    \end{lemma}
    \begin{proof}
        Since $H$ is closed, we can write $H=\bigcap_iU_i$ as the intersection of all open subgroups $U_i\subseq_oF$ containing $H$. For each $i$, let $f_i:\mc{S}(H)\to\mc{S}(U_i)$ be the map of graphs obtained by extending the set of generators of $H$ to one for $U_i$, and constructing $\mc{S}(H)$ inside $\mc{S}(U_i)$ accordingly. Note that these maps \textit{need not} be injective, as the construction of $\mc{S}(U_i)$ might fold more admissible edges than $\mc{S}(H)$ does.

        As $\mc{S}(H)$ is finite, there are finitely many equivalence relations on $\mc{S}(H)$, so, after reordering, there exists $m\in\N$ such that for any $h,h'\in H$, we have $h=h'$ iff $f_i(h)=f_i(h')$ for all $1\leq i\leq m$. Let $U\coloneqq\bigcap_{i\leq m}U_i$, which is open, and let $f:\mc{S}(H)\to\mc{S}(U)$ be as above. Then $f$ is injective, so $\mc{S}(H)$ expands a covering $\mc{S}(U)$ whose group of deck transformations lies in $\mathbf{V}$ since $\Aut(\mc{S}(U))\iso F/\pi_1(\mc{S}(U))_F=F/U_F\in\mathbf{V}$.
    \end{proof}

    \section{Proof of main theorem}

    \subsection{Setup}

    Let us attempt to implement the strategy outlined in Remark \ref{rem:strategy} to prove Theorem \ref{mthm}. We will see how these $p$-elementary extensions, and in particular (\ref{eqn:magic}), naturally arise in the proof.

    Let $\mathbf{V}$ be a pseudovariety of groups closed under extensions with abelian kernel and let $H_1,\dots,H_n\subseq F$ be finitely generated subgroups which are closed in the pro-$\mathbf{V}$ topology of $F$. We seek, for each word $w\in F$, a group $K\in\mathbf{V}$ such that if $[w]_K\in[H_1\cdots H_n]_K$, then $w\in H_1\cdots H_n$; Lemma \ref{lem:closed} then finishes the proof.

    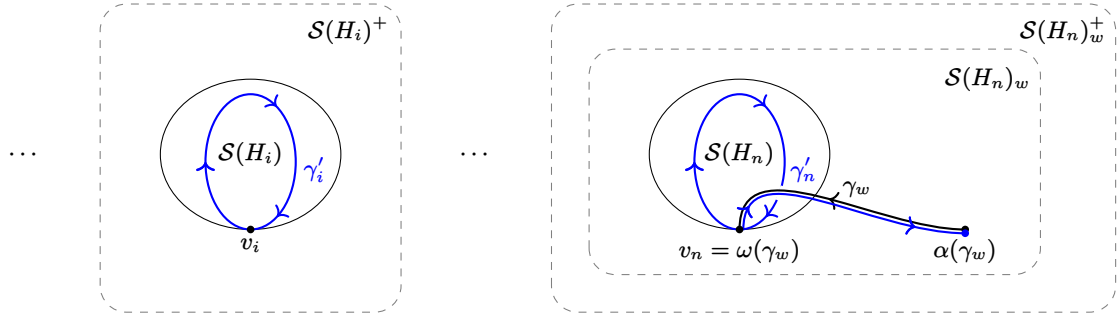
\begin{figure}
        \center
        \begin{center}
            \begin{tikzpicture}
                \draw[dashed, gray, rounded corners=10] (-2,-1.1) rectangle (2,3);
                \node at (1.3,2.65) {\small$\mc{S}(H_i)^+$};

                \node at (0,1) {\small$\mc{S}(H_i)$};
                \draw (0,1) ellipse (1.2cm and 1cm);

                \draw[thick, blue, decoration = {markings, mark=between positions 0 and 1 step 0.333 with \arrow{>}}, postaction = {decorate}]
                    (0,0.9) ellipse (-0.6cm and 0.9cm);
                \node[blue] at (0.85,0.8) {\small$\gamma_i'$};

                \fill (0,0) circle (0.02in) node[below] {\small$v_i$};

                \node at (-3,1) {\small$\cdots$};
                \node at (3,1) {\small$\cdots$};

                \begin{scope}[xshift=6.5cm]
                    \draw[dashed, gray, rounded corners=10] (-2.5,-1.1) rectangle (5,3);
                    \node at (4.3,2.65) {\small$\mc{S}(H_n)_w^+$};

                    \draw[dashed, gray, rounded corners=10] (-2,-0.6) rectangle (4,2.4);
                    \node at (3.3,2) {\small$\mc{S}(H_n)_w$};

                    \node at (0,1) {\small$\mc{S}(H_n)$};
                    \draw (0,1) ellipse (1.2cm and 1cm);

                    \draw[thick, blue, decoration = {markings, mark=between positions 0 and 1 step 0.333 with \arrow{>}}, postaction = {decorate}]
                        (0,0.9) ellipse (-0.6cm and 0.9cm);
                    \node[blue] at (0.85,0.8) {\small$\gamma_n'$};

                    \coordinate (a) at (0.478,0.4);
                    \fill[white, rotate=-20] (a) rectangle ($(a)+(0.05,0.2)$);

                    \draw[->-=0.55, thick] (3,0) to[out=180, in=90] node[xshift=0.5cm, yshift=0.1cm] {\small$\gamma_w$} (0,0);
                    \fill (3,0) circle (0.02in) node[below] {\small$\alpha(\gamma_w)$};

                    \draw[thick, blue, ->-=0.1, ->-=0.8] (0.05,-0.008) .. controls (0.05,1.1) and (1.75,-0.05) .. (3,-0.05);
                    \fill[blue] (3,-0.05) circle (0.02in);

                    \fill (0,0) circle (0.02in) node[below] {\small$v_n=\omega(\gamma_w)$};
                \end{scope}
            \end{tikzpicture}

            \hspace{0.05in}
            \begin{minipage}[t]{0.35\textwidth}
                \subcaption{The Stallings graph $\mc{S}(H_i)$ and its expansion $\mc{S}(H_i)^+$, together with a loop $\gamma_i'$ based at $v_i$ lying inside $\mc{S}(H_i)$.}
            \end{minipage}
            \hspace{0.45in}
            \begin{minipage}[t]{0.45\textwidth}
                \subcaption{The Stallings graph $\mc{S}(H_n)$, expanded to a graph $\mc{S}(H_n)_w$ by attaching a path $\gamma_w$. The path $\gamma_n'$ traces a loop at $v_n=\omega(\gamma_w)$, and then follows $\bar{\gamma_w}$ to $\alpha(\gamma_w)$.}
            \end{minipage}
        \end{center}
        \begin{minipage}{0.95\textwidth}
            \caption{The Stallings graphs associated to $H_1,\dots,H_n$ and $w$ and the loops $\gamma_1',\dots,\gamma_n'$ representing the identity in $G$, which we need to modify so that they represent the identity in $F$.}
            \label{fig:proof_balloons}
            \vspace{-0.1in}
        \end{minipage}
    \end{figure}

    For each $1\leq i<n$, let $\Gamma_i\coloneqq(\mc{S}(H_i),v_i)$ be the Stallings graph of $H_i$, and let $\Gamma_n\coloneqq\mc{S}(H_n)_w$ be the graph obtained by attaching to $(\mc{S}(H_n),v_n)$ a path $\gamma_w$ labelled $w$, where $\omega(\gamma_w)=v_n$ and $\alpha(\gamma_w)\not\in\mc{S}(H_n)$, which we then fold as in the proof of Hall's Theorem; see Figure \ref{fig:proof_balloons}. By Lemma \ref{lem:deck}, these graphs extend to covers whose groups of deck transformations $G_1,\dots,G_n$ lie in $\mathbf{V}$. Let $\Gamma\coloneqq\bigsqcup_{i=1}^n\Gamma_i$ and note that the group of deck transformations of $\Gamma^+=\bigsqcup_{i=1}^n\Gamma_i^+$ can be realized as the group
    \vspace{-0.05in}
    \begin{equation*}
        G\coloneqq\l\langle(f^1_x,\dots,f^n_x)\st x\in X\r\rangle\subseq G_1\times\cdots\times G_n,
        \vspace{-0.05in}
    \end{equation*}
    where $f^i_x\in G_i$ is the deck transformation induced by the action of $x$. Since $\mathbf{V}$ is closed under subgroups and finite direct products, we see that $G\in\mathbf{V}$.
    
    This is where we diverge from Stallings' proof of Hall's Theorem above. To see why, let us first attempt to show that our candidate group $G$ works, so assume that $[w]_G\in[H_1\cdots H_n]_G$. To show that $w\in H_1\cdots H_n$, which is equivalent to $1\in H_1\cdots H_nw^{-1}$, it suffices to construct paths $\gamma_i$ in $\Gamma_i$ for each $1\leq i\leq n$, such that the following two properties hold.
    \begin{enumerate}
        \item $\gamma_i$ is a loop based at $v_i$ for each $1\leq i<n$, and $\gamma_n$ is a path from $\omega(\gamma_w)=v$ to $\alpha(\gamma_w)$.
            \vspace{-0.05in}
        \item $[\gamma_1\cdots\gamma_n]_F=1$.
    \end{enumerate}

    Paths satisfying (1) are easy to come by: since $[w]_G\in[H_1\cdots H_n]_G$, we can write $[w]_G=[h_1\cdots h_n]_G$ for some $h_i\in H_i$, so $[h_1\cdots h_nw^{-1}]_G=1$, and hence the reduced paths $\gamma_i'$ in $\Gamma_i$ starting from $v_i$ with $\ell(\gamma_i')=h_i$ for $1\leq i<n$ and $\ell(\gamma_n')=h_nw^{-1}$ satisfy (1); see Figure \ref{fig:proof_balloons}. Note that $[\gamma_1'\cdots\gamma_n']_G=1$ by construction, but there is no reason for the labels of $\gamma_1',\dots,\gamma_n'$ to cancel out in $F$, as $G$ is not necessarily free.

    The issue is that the words $w_i\coloneqq\ell(\gamma_i')$ might not have sufficient overlap with its adjacent words to cancel in $F$. In other words, the paths $\eta_i'$ in the Cayley graph $\mc{C}(G)$, labelled by $w_i$ and placed tip-to-tail starting from $1$, forms a homotopically non-trivial loop; our task is to pinch this loop down (see Figure \ref{fig:proof_to_pinch}).

    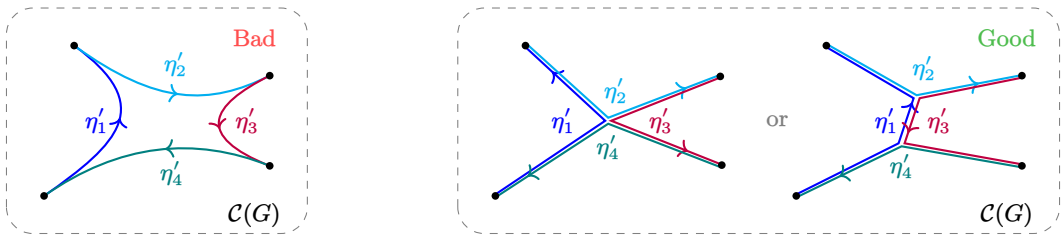
\begin{figure}
        \center
        \begin{center}
            \begin{tikzpicture}
                \coordinate (a) at (0,0);
                \coordinate (b) at (0.4,2);
                \coordinate (c) at (3,1.6);
                \coordinate (d) at (3,0.4);

                \coordinate (ab) at (1,0.7);   
                \coordinate (bc) at (1.5,1.2); 
                \coordinate (cd) at (2.2,1.2); 
                \coordinate (da) at (2,0.8);   

                \draw[->-=0.55, thick, blue]   (a) .. controls (ab) and (bc) .. (b);
                \draw[->-=0.55, thick, cyan]   (b) .. controls (bc) and (cd) .. (c);
                \draw[->-=0.55, thick, purple] (c) .. controls (cd) and (da) .. (d);
                \draw[->-=0.45, thick, teal]   (d) .. controls (da) and (ab) .. (a);

                \fill (a) circle (0.02in);
                \fill (b) circle (0.02in);
                \fill (c) circle (0.02in);
                \fill (d) circle (0.02in);

                \node[blue]   at (0.7,1)     {\small$\eta_1'$};
                \node[cyan]   at (1.75,1.75) {\small$\eta_2'$};
                \node[purple] at (2.7,1)     {\small$\eta_3'$};
                \node[teal]   at (1.7,0.3)   {\small$\eta_4'$};

                \node[red, opacity=0.7] at (2.8,2.1) {\small Bad};
                \node at (2.8,-0.25) {\small$\mc{C}(G)$};

                \def\s{6}
                \def\t{10}
                \def\p{0.04}

                \draw[rounded corners=10, dashed, gray] (-0.5,-0.5) rectangle (3.5,2.5);
                \draw[rounded corners=10, dashed, gray] (\s-0.5,-0.5) rectangle (\t+3.5,2.5);

                \begin{scope}[xshift=\s cm]
                    \coordinate (center) at (1.5,1); 

                    \coordinate (newA) at ($(a)+(\s,0)$);
                    \coordinate (newB) at ($(b)+(\s,0)$);
                    \coordinate (newC) at ($(c)+(\s,0)$);
                    \coordinate (newD) at ($(d)+(\s,0)$);

                    \draw[thick, blue, ->-=0.85]   ($(newA)+(0,\p/2)$)  -- ($(center)+(-\p,0)$) -- ($(newB)+(0,-\p/2)$);
                    \draw[thick, cyan, ->-=0.85]   ($(newB)+(\p,0)$)    -- ($(center)+(0,\p)$)  -- ($(newC)+(-\p,0)$);
                    \draw[thick, purple, ->-=0.85] ($(newC)+(0,-\p/2)$) -- ($(center)+(\p,0)$)  -- ($(newD)+(0,\p/2)$);
                    \draw[thick, teal, ->-=0.85]   ($(newD)+(0,-\p/2)$) -- ($(center)+(0,-\p)$) -- ($(newA)+(0,-\p/2)$);

                    \fill ($(newA)$)               circle (0.02in);
                    \fill ($(newB)$)               circle (0.02in);
                    \fill ($(newC)+(-0.01,-0.01)$) circle (0.02in);
                    \fill ($(newD)+(0.01,0.01)$)   circle (0.02in);

                    \node[blue]   at (0.9,1)    {\small$\eta_1'$};
                    \node[cyan]   at (1.6,1.4)  {\small$\eta_2'$};
                    \node[purple] at (2.2,1)    {\small$\eta_3'$};
                    \node[teal]   at (1.5,0.65) {\small$\eta_4'$};

                    \node[darkGreen, opacity=0.7] at (6.8,2.1) {\small Good};
                    \node at (6.8,-0.25) {\small$\mc{C}(G)$};
                \end{scope}

                \begin{scope}[xshift=\t cm]
                    \coordinate (centerBot) at (1.4,0.7); 
                    \coordinate (centerTop) at (1.6,1.3); 

                    \coordinate (newA) at ($(a)+(\t,0)$);
                    \coordinate (newB) at ($(b)+(\t,0)$);
                    \coordinate (newC) at ($(c)+(\t,0)$);
                    \coordinate (newD) at ($(d)+(\t,0)$);

                    \draw[thick, blue, ->-=0.6] ($(newA)+(0,\p/2)$)       -- ($(centerBot)+(-\p,0)$) -- ($(centerTop)+(-\p,0)$) -- ($(newB)+(0,-\p/2)$);
                    \draw[thick, cyan, ->-=0.8] ($(newB)+(\p,0)$)         -- ($(centerTop)+(0,\p)$)                             -- ($(newC)+(-\p,0)$);
                    \draw[thick, purple, ->-=0.52] ($(newC)+(0,-2*\p/3)$) -- ($(centerTop)+(\p,0)$)  -- ($(centerBot)+(\p,0)$)  -- ($(newD)+(0,\p/2)$);
                    \draw[thick, teal, ->-=0.8] ($(newD)+(0,-\p/2)$)      -- ($(centerBot)+(0,-\p)$)                            -- ($(newA)+(0,-\p/2)$);

                    \fill ($(newA)$) circle (0.02in);
                    \fill ($(newB)$) circle (0.02in);
                    \fill ($(newC)$) circle (0.02in);
                    \fill ($(newD)$) circle (0.02in);

                    \node[blue]   at (1.2,1)   {\small$\eta_1'$};
                    \node[cyan]   at (1.7,1.7) {\small$\eta_2'$};
                    \node[purple] at (1.9,1)   {\small$\eta_3'$};
                    \node[teal]   at (1.4,0.4) {\small$\eta_4'$};

                    \node[gray] at (-0.25,1) {\small or};
                \end{scope}
            \end{tikzpicture}
        \end{center}
        \begin{minipage}{0.95\textwidth}
            \vspace{0.1in}
            \caption{In $\mc{C}(G)$, the words $w_1,\dots,w_n$ trace out paths $\eta_i'$ starting from $1$ that together closes a loop; we need to ensure that this loop is homotopically trivial, so that $[w_1\cdots w_n]_F=1$.}
            \label{fig:proof_to_pinch}
        \end{minipage}
    \end{figure}

    First, however, we need to clarify how paths in $\mc{C}(G)$ can be used to construct paths in $\Gamma$. We will state a lemma to this effect in a slightly more general form, which will be needed in the sequel.

    \begin{lemma}\label{lem:create_paths}
        Let $\Gamma$ be an immersion and let $G$ be the group of deck transformations of any covering expansion $\Gamma^+$ of $\Gamma$. Let $H$ be a group and fix a morphism $\phi:H\to G$ such that $[w]_G=\phi([w]_H)$ for each word $w\in F$.

        Let $\eta$ and $\eta'$ be reduced paths in $\mc{C}(H)$ such that $\alpha(\eta)=\alpha(\eta')$, and suppose that each geometric edge of $\eta$ is also traversed by $\eta'$. For any reduced path $\gamma'$ in $\Gamma$ such that $\ell(\gamma')=\ell(\eta')$, there is a path $\gamma$ in $\Gamma$ such that $\ell(\gamma)=\ell(\eta)$ and $\alpha(\gamma)=\alpha(\gamma')$. If in addition $\omega(\eta)=\omega(\eta')$, then $\omega(\gamma)=\omega(\gamma')$ too.
    \end{lemma}
    \begin{proof}
        Consider the covering map $f:(\mc{C}(H),\alpha(\eta'))\onto(\Gamma^+,\alpha(\gamma'))$, which exists since $\pi_1(\mc{C}(H),\alpha(\eta'))$ is a subgroup of $\pi_1(\Gamma^+,\alpha(\gamma'))$. Since $\eta'$ and $\gamma'$ are reduced paths with the same label and $\alpha(f(\eta'))=f(\alpha(\eta'))=\alpha(\gamma')$, we see that $\gamma'=f(\eta')$. Set $\gamma\coloneqq f(\eta)$, which lies in $\Gamma$ since $\gamma'=f(\eta')$ lies in $\Gamma$ and every geometric edge of $\eta$ is also traversed by $\eta'$. The rest of the claims are easily checked.
    \end{proof}

    \begin{figure}
        \center
        \begin{center}
            \begin{tikzpicture}
                \def\s{6}
                \def\p{0.04}
                \draw[rounded corners=10, dashed, gray] (-0.5,-0.5) rectangle (4,2.5);
                \draw[rounded corners=10, dashed, gray] (\s-0.5,-0.5) rectangle (\s+4,2.5);

                \coordinate (a) at (0,0);
                \coordinate (b) at (1,0.5);
                \coordinate (c) at (2,1);
                \coordinate (d) at (2.5,1.5);
                \coordinate (e) at (3,2);

                \coordinate (b1) at (1.2,2); 
                \coordinate (b2) at (0,0.1); 

                \draw[line width=8pt, line cap=round, rounded corners=1, pink, opacity=0.6] (a) to[out=10, in=260] (b);

                \draw[thick, blue, ->-=0.5] (a) to[out=10, in=260] (b);
                \draw[thick, blue, ->-=0.6] (b) .. controls (b1) and (b2) .. (b);
                \draw[thick, blue, ->-=0.35, ->-=0.7] (b) -- (c) -- (e);

                \draw[thick, cyan, ->-=0.35, ->-=0.6] ($(e)+(-\p/2,\p)$)
                    -- ($(d)+(-\p/2,\p)$)
                    to[out=170, in=90] ($(c)+(0,\p)$)
                    to[out=270, in=-30] ($(b)+(\p,0)$)
                    to[out=265, in=10] ($(a)+(0,-\p)$);

                \fill ($(a)+(0,-\p/2)$)    circle (0.02in) node[below] {\small$1$} node[yshift=0.4cm, pink] {\small$\Omega$};
                \fill ($(e)+(-\p/2,\p/2)$) circle (0.02in) node[right] {\small$[w_1]_G$};

                \node[blue] at (0.5,1.25) {\small$\eta_1'$};
                \node[cyan] at (2,0.5) {\small$\eta_2'$};

                \draw[thick, red, ->-=0.65] (b) -- node[above] {\small$e$} ($(b)+(0.4,0.2)$);
                \fill[red] (b) circle (0.02in);
                \fill[red] ($(b)+(0.4,0.2)$) circle (0.02in);

                \node[red, opacity=0.7] at (0,2.1) {\small Bad};
                \node at (3.3,-0.25) {\small$\mc{C}(G)$};

                \begin{scope}[xshift=\s cm]
                    \coordinate (newA) at ($(a)+(\s,0)$);
                    \coordinate (newB) at ($(b)+(\s,0)$);
                    \coordinate (newC) at ($(c)+(\s,0)$);
                    \coordinate (newD) at ($(d)+(\s,0)$);
                    \coordinate (newE) at ($(e)+(\s,0)$);

                    \coordinate (newB1) at ($(b1)+(\s,0)+(0,-\p)$);
                    \coordinate (newB2) at ($(b2)+(\s,0)+(-\p,-\p)$);

                    \draw[line width=8pt, line cap=round, rounded corners=1, pink, opacity=0.6] (newA) to[out=10, in=260] (newB) to[out=-30, in=270] ($(newC)+(\p,\p)$) to[out=90, in=170] (newD) -- (newE);

                    \draw[thick, blue, ->-=0.5] (newA) to[out=10, in=260] (newB);
                    \draw[thick, blue, ->-=0.6] (newB) .. controls (newB1) and (newB2) .. ($(newB)+(\p,-\p)$);
                    \draw[thick, blue, ->-=0.35, ->-=0.7] ($(newB)+(\p,-\p)$) to[out=-30, in=270] ($(newC)+(\p,\p)$) to[out=90, in=170] (newD) -- (newE);
                    \draw[thick, blue, opacity=0.2] ($(newB)+(\p,-\p)$) -- (newC) -- (newD);

                    \draw[thick, cyan, ->-=0.35, ->-=0.6] ($(newE)+(-\p/2,\p)$)
                        -- ($(newD)+(-\p/2,\p)$)
                        to[out=170, in=90] ($(newC)+(0,\p)$)
                        to[out=270, in=-30] ($(newB)+(\p,0)$)
                        to[out=265, in=10] ($(newA)+(0,-\p)$);

                    \fill ($(newA)+(0,-\p/2)$)    circle (0.02in) node[below] {\small$1$};
                    \fill ($(newE)+(-\p/2,\p/2)$) circle (0.02in) node[right] {\small$[w_1]_G$};

                    \node[blue] at (0.6,1.2) {\small$\eta_1'$};
                    \node[cyan] at (1.8,0.3) {\small$\eta_2'$};
                    \node[pink, yshift=0.35cm] at (newA) {\small$\eta$};
                    \node[pink, xshift=0.3cm, yshift=0.35cm, rotate=10] at (newA) {\small$\rightsquigarrow$};

                    \node[darkGreen, opacity=0.7] at (0.1,2.1) {\small Good};
                    \node at (3.3,-0.25) {\small$\mc{C}(G)$};
                \end{scope}
            \end{tikzpicture}
        \end{center}
        \begin{minipage}{0.95\textwidth}
            \vspace{0.1in}
            \caption{For the case $n=2$, the paths $\eta_1':1\to[w_1]_G$ and $\eta_2':[w_1]_G\to1$ cannot enclose bigons if we take $K\coloneqq G^{\mathbf{Ab}(p)}$; since $[w_1w_2]_K=1$, every edge in $\Delta_1\cup\Delta_2$ must be traversed $0$ times mod $p$. In the figure, this is the case if the self-loop that $\eta_1'$ makes is traversed a multiple of $p$ times.}
            \label{fig:proof_2}
        \end{minipage}
    \end{figure}
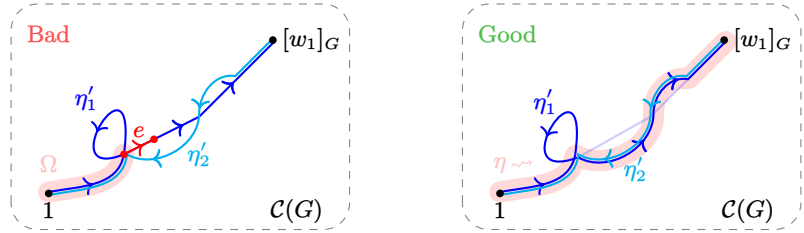

    \subsection{The proof in the case $n=2$}

    Let us first see Lemma \ref{lem:create_paths} in action for the case $n=2$. Here, the words $w_1$ and $w_2$ trace out paths $\eta_1':1\to[w_1]_G$ and $\eta_2':[w_1]_G\to1$ in $\mc{C}(G)$, which enclose zero or more bigons; see Figure \ref{fig:proof_2}. To ensure that no bigons are enclosed, and so that they admit a common `spine' $\eta:1\to[w_1]_G$, it suffices to `force' certain edges in $\mc{C}(G)$ to be traversed by both $\eta_1'$ and $\eta_2'$.

    By virtue of (\ref{eqn:magic}), this is where extensions of the form $K\coloneqq G^{\mathbf{Ab}(p)}$ come in; recall that Lemma \ref{lem:extension} ensures that is a prime $p$ such that $K\in\mathbf{V}$. Under the assumption that $[w]_K\in[H_1H_2]_K$, we can construct paths $\gamma_i'$ in $\Gamma_i$ as above, but this time we have $[\gamma_1'\gamma_2']_K=1$, and hence $[w_1w_2]_K=1$; as $K\onto G$, clearly $[w_1w_2]_G=1$ too, and so there are paths $\eta_1':1\to[w_1]_G$ and $\eta_2':[w_1]_G\to1$ in $\mc{C}(G)$ such that $\ell(\eta_i')=w_i$. Now, we can show that they enclose no bigons. To this end, let $\Delta_i$ be the subgraphs of $\mc{C}(G)$ spanned by $\eta_i'$ for $i=1,2$.

    \begin{center}
        \begin{minipage}{0.95\textwidth}
            \begin{claim*}
                The intersection $\Delta_1\cap\Delta_2$ contains a reduced path $\eta:1\to[w_1]_G$.
            \end{claim*}
            \begin{proof}
                Suppose otherwise and let $\Omega\sub\Delta_1\cap\Delta_2$ be the connected component containing $1$, so $\eta_1'$ leaves $\Omega$ exactly once more than it enters; see Figure \ref{fig:proof_2}. Formally, let $(e_n\st n\leq|\eta_1'|)$ list the edges of $\eta_1'$ in order, possibly with repetition, and let
                \vspace{-0.05in}
                \begin{equation*}
                    \begin{aligned}
                        \Omega_\mathrm{in}&\coloneqq\{n\st e_n\in\Delta_1,\alpha(e_n)\in\Delta_1\comp\Delta_2,\textrm{ and }\omega(e_n)\in\Omega\} \\
                        \Omega_\mathrm{out}&\coloneqq\{n\st e_n\in\Delta_1,\omega(e_n)\in\Delta_1\comp\Delta_2,\textrm{ and }\alpha(e_n)\in\Omega\},
                    \end{aligned}
                    \vspace{-0.05in}
                \end{equation*}
                from which it can be seen that $|\Omega_\mathrm{out}|-|\Omega_\mathrm{in}|=1$. Let $\Omega_\blob^+$ denote (the index set of) the set of positively-oriented edges in $\Omega_\blob$, where $\blob\in\l\{\mathrm{in},\mathrm{out}\r\}$, so $|\Omega_\blob|\equiv\sum\nolimits_{n\in\Omega_\blob^+}[w_1(e_n)]_p$ mod $p$. Thus we have
                \vspace{-0.1in}
                \begin{equation*}
                    1=|\Omega_\mathrm{out}|-|\Omega_\mathrm{in}|\equiv\l(\sum\nolimits_{n\in\Omega_\mathrm{out}^+}[w_1(e_n)]_p-\sum\nolimits_{n\in\Omega_\mathrm{in}^+}[w_1(e_n)]_p\r)\ \textrm{mod }p,
                    \vspace{-0.05in}
                \end{equation*}
                so there is an edge $e\coloneqq e_n$ with $n\in\Omega_\mathrm{out}^+\cup\Omega_\mathrm{in}^+$ such that $[w_1(e)]_p\neq0$. As $e\not\in\Delta_2$, we have $[w_2(e)]_p=0$, and so $[w_1w_2(e)]_p=[w_1(e)]_p\neq0$. This is a contradiction since on the one hand $[w_1w_2]_K$ is trivial, but according to (\ref{eqn:magic}), the element $[w_1w_2]_K$ has a non-trivial summand $([w_1w_2(e)]_p,1)$.
                \qedlem
            \end{proof}
        \end{minipage}
    \end{center}

    We now invoke Lemma \ref{lem:create_paths} to project $\eta$ and $\bar{\eta}$ down to paths $\gamma_1$ and $\gamma_2$ in $\Gamma$ in such a way that (1) remains valid, and now clearly $[\gamma_1\gamma_2]_F=[\eta]_F[\bar{\eta}]_F=1$. This finishes the proof of Theorem \ref{mthm} in the case $n=2$.

    \subsection{The general case}

    In the general case, we need to work a bit more, but the overall idea remains the same: the words $w_i$ for $1\leq i\leq n$ trace out paths $\eta_i'$ in $\mc{C}(G)$, and we need to pass to an extension to pinch down any $m$-gons, where $2\leq m\leq n$, that they enclose. However, one complication arises: the subgraphs $\Delta_1,\dots,\Delta_n$ of $\mc{C}(G)$ that the paths $\eta_1',\dots,\eta_n'$ span might not even have a common intersection; see Figure \ref{fig:proof_to_pinch}.

    This\vspace{-0.023in} issue can be (partially) resolved if we replace the paths $\eta_i'$ in $\mc{C}(G)$ by paths in the Cayley graph of an extension of $G$. Thus, inductively, we are led to considering an iterated extension $G_0\coloneqq G$ and $G_i\coloneqq G_{i-1}^{\mathbf{Ab}(p_i)}$ for $1\leq i<n$. We claim that taking $K\coloneqq G_{n-1}$ works. Recall that Lemma \ref{lem:extension} ensures that there is a tuple $(p_1,\dots,p_{n-1})$ of primes such that $K\in\mathbf{V}$, so we are still in good shape.

    Let $K\in\mathbf{V}$ be the iterated extension of $G$ as above and suppose that $[w]_K\in[H_1\cdots H_n]_K$, so as before we can construct paths $\gamma_i'$ in $\Gamma_i$ such that $[\gamma_1'\cdots\gamma_n']_K=1$. Let $w_i\coloneqq\ell(\gamma_i')$ for each $1\leq i\leq n$, so $[w_1\cdots w_n]_K=1$. Set $g_0\coloneqq1$ and $g_i\coloneqq[w_1\cdots w_i]_{G_{n-2}}$ for $1\leq i\leq n$, and let $\Delta_i$ be the subgraph of $\mc{C}(G_{n-2})$ spanned by the path $\eta_i':g_{i-1}\to g_i$ labelled $w_i$. Since $g_n=[w_1\cdots w_n]_K=1$, we see that $\alpha(\eta_1')=1=\omega(\eta_n')$. As in the case $n=2$, let $\Omega$ denote the connected component of the identity of $\Delta_1\cap\Delta_n$.

    \begin{center}
        \begin{minipage}{0.95\textwidth}
            \begin{claim*}
                There exists $1<j<n$ such that $\Delta_j\cap\Omega\neq\e$.
            \end{claim*}
            \begin{proof}
                Suppose otherwise that $\Delta_j\cap\Omega=\e$ for all $1<j<n$. Since $g_1\in\Delta_2$, we see that $g_1\in\Delta_1\comp\Omega$, and hence arguing as in the case $n=2$ furnishes a positively-oriented edge $e$ with one vertex in $\Omega$ such that $[w_1(e)]_p\neq0$ and $e\in\Delta_1\comp\Delta_n$, where $p\coloneqq p_{n-1}$; see Figure \ref{fig:proof_n}.

                \hspace{0.2in}By hypothesis, we have $e\not\in\Delta_j$ for every $1<j<n$, so $[w_j(e)]_p=0$ for every $1<j<n$, and hence $[w_1\cdots w_n(e)]_p=[w_1(e)]_p\neq0$. But we chose $K$ to be a $p$-elementary extension of $G_{n-2}$, from which it follows by virtue of (\ref{eqn:magic}) that $[w_1\cdots w_n]_K$ is non-trivial, a contradiction.
                \qedlem
            \end{proof}
        \end{minipage}
    \end{center}

    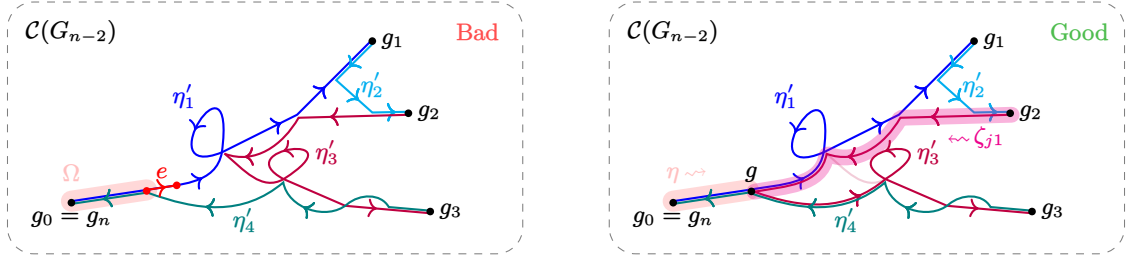
\begin{figure}
        \center
        \begin{center}
            \begin{tikzpicture}
                \def\s{8}
                \def\p{0.04}
                \draw[rounded corners=10, dashed, gray] (-1.85,-0.85) rectangle (5,2.5);
                \draw[rounded corners=10, dashed, gray] (\s-1.85,-0.85) rectangle (\s+5,2.5);

                \coordinate (i) at (-1,-0.15);
                \coordinate (a) at (0,0);
                \coordinate (b) at (1,0.5);
                \coordinate (c) at (2,1);
                \coordinate (d) at (2.5,1.5);
                \coordinate (e) at (3,2);
                \coordinate (f) at (3,1);
                \coordinate (g) at (3.5,1);

                \coordinate (b1) at (1.2,2);
                \coordinate (b2) at (0,0.1);

                \draw[line width=8pt, line cap=round, rounded corners=1, pink, opacity=0.6] (i) -- (a);

                \draw[thick, blue, ->-=0.8] (i) -- (a) to[out=10, in=260] (b);
                \draw[thick, blue, ->-=0.6] (b) .. controls (b1) and (b2) .. (b);
                \draw[thick, blue, ->-=0.35, ->-=0.65] (b) -- (c) -- (e);

                \draw[thick, cyan, ->-=0.25, ->-=0.6, ->-=0.9] ($(e)+(\p/2,-\p)$) -- ($(d)+(\p/2,-\p)$) -- ($(f)+(0,0.75*\p)$) -- ($(g)+(0,\p)$);
                \draw[thick, purple, ->-=0.4, ->-=0.9] (g) -- ($(c)+(\p/2,-\p)$) to[out=240, in=-20] ($(b)+(\p,-\p/2)$);

                \begin{scope}[shift={(b)}, rotate=-50, scale=0.8]
                    \coordinate (rotA) at (0,0);
                    \coordinate (rotB) at (1,0.5);
                    \coordinate (rotC) at (2,1);
                    \coordinate (rotD) at (2.5,1.5);
                    \coordinate (rotE) at (3,2);

                    \coordinate (rotB1) at (1.2,2);
                    \coordinate (rotB2) at (0,0.1);

                    \draw[thick, purple] ($(rotA)+(\p,\p/2)$) to[out=10, in=260] (rotB);
                    \draw[thick, purple, ->-=0.6] (rotB) .. controls (rotB1) and (rotB2) .. (rotB);
                    \draw[thick, purple, ->-=0.65] (rotB) -- (rotC) -- (rotE);

                    \draw[thick, teal, ->-=0.35, ->-=0.65, ->-=0.92] ($(rotE)+(-\p/2,\p)$)
                        -- ($(rotD)+(-\p/2,\p)$)
                        to[out=170, in=90] ($(rotC)+(0,\p)$)
                        to[out=270, in=-30] ($(rotB)+(\p,-\p)$)
                        to[out=270, in=30] ($(a)+(\p,-\p/2)$)
                        -- ($(i)+(\p,-\p/2)$);

                    \fill ($(rotE)+(-\p/2,\p/2)$) circle (0.025in) node[right] {\small$g_3$};
                \end{scope}

                \fill ($(i)+(0,-\p/2)$)    circle (0.02in) node[below] {\small$g_0=g_n$} node[yshift=0.4cm, pink] {\small$\Omega$};
                \fill ($(e)+(0,-\p/2)$)    circle (0.02in) node[right] {\small$g_1$};
                \fill ($(g)+(-\p/2,\p/2)$) circle (0.02in) node[right] {\small$g_2$};

                \node[blue]   at (0.5,1.25) {\small$\eta_1'$};
                \node[cyan]   at (3,1.4)    {\small$\eta_2'$};
                \node[purple] at (2.4,0.5)  {\small$\eta_3'$};
                \node[teal]   at (1.3,-0.4) {\small$\eta_4'$};

                \draw[thick, red, ->-=0.65] (a) -- node[above] {\small$e$} ($(a)+(0.4,0.061)$);
                \fill[red] ($(a)+(0,-\p/3)$) circle (0.02in);
                \fill[red] ($(a)+(0.4,0.061)$) circle (0.02in);

                \node[red, opacity=0.7] at (4.4,2.1) {\small Bad};
                \node at (-1,2.1) {\small$\mc{C}(G_{n-2})$};

                \begin{scope}[xshift=\s cm]
                    \coordinate (newI) at ($(i)+(\s,0)$);
                    \coordinate (newA) at ($(a)+(\s,0)$);
                    \coordinate (newB) at ($(b)+(\s,0)$);
                    \coordinate (newC) at ($(c)+(\s,0)$);
                    \coordinate (newD) at ($(d)+(\s,0)$);
                    \coordinate (newE) at ($(e)+(\s,0)$);
                    \coordinate (newF) at ($(f)+(\s,0)$);
                    \coordinate (newG) at ($(g)+(\s,0)$);

                    \coordinate (newB1) at ($(b1)+(\s,0)+(0,-\p)$);
                    \coordinate (newB2) at ($(b2)+(\s,0)+(-\p,-\p)$);

                    \draw[line width=8pt, line cap=round, rounded corners=1, pink, opacity=0.6] (newI) -- (newA);

                    \draw[thick, blue, ->-=0.8] (newI) -- (newA) to[out=10, in=260] (newB);
                    \draw[thick, blue, ->-=0.6] (newB) .. controls (newB1) and (newB2) .. (newB);
                    \draw[thick, blue, ->-=0.35, ->-=0.65] (newB) -- (newC) -- (newE);

                    \draw[thick, cyan, ->-=0.25, ->-=0.6, ->-=0.9] ($(newE)+(\p/2,-\p)$) -- ($(newD)+(\p/2,-\p)$) -- ($(newF)+(0,0.75*\p)$) -- ($(newG)+(0,\p)$);
                    \draw[thick, purple, ->-=0.19, ->-=0.42, ->-=0.93] (newG) -- ($(newC)+(\p/2,-\p)$) to[out=240, in=-20] ($(newB)+(\p,-\p/2)$) to[out=260, in=10] ($(newA)+(2*\p,-0.6*\p)$) to[out=-20, in=220] (1.83,0.15);

                    \draw[line width=6pt, line cap=round, rounded corners=1, magenta, opacity=0.3]
                        (newG) -- ($(newC)+(\p/2,-\p)$) to[out=240, in=-20] ($(newB)+(\p,-\p/2)$) to[out=260, in=10] ($(newA)+(2*\p,-0.6*\p)$);

                    \begin{scope}[shift={(newB)}, rotate=-50, scale=0.8]
                        \coordinate (rotA) at (0,0);
                        \coordinate (rotB) at (1,0.5);
                        \coordinate (rotC) at (2,1);
                        \coordinate (rotD) at (2.5,1.5);
                        \coordinate (rotE) at (3,2);

                        \coordinate (rotB1) at (1.2,2);
                        \coordinate (rotB2) at (0,0.1);

                        \draw[thick, purple, opacity=0.2] ($(rotA)+(\p,\p/2)$) to[out=10, in=260] (rotB);
                        \draw[thick, purple, ->-=0.6] (rotB) .. controls (rotB1) and (rotB2) .. (rotB);
                        \draw[thick, purple, ->-=0.65] (rotB) -- (rotC) -- (rotE);

                        \draw[thick, teal, ->-=0.35, ->-=0.65, ->-=0.92] ($(rotE)+(-\p/2,\p)$)
                            -- ($(rotD)+(-\p/2,\p)$)
                            to[out=170, in=90] ($(rotC)+(0,\p)$)
                            to[out=270, in=-30] ($(rotB)+(\p,-\p)$)
                            to[out=270, in=30] ($(newA)+(\p,-\p/2)$)
                            -- ($(newI)+(\p,-\p/2)$);

                        \fill ($(rotE)+(-\p/2,\p/2)$) circle (0.025in) node[right]{\small$g_3$};
                    \end{scope}

                    \fill ($(newI)+(0,-\p/2)$)    circle (0.02in) node[below] {\small$g_0=g_n$};
                    \fill ($(newA)+(\p,-\p/2)$)   circle (0.02in) node[above] {\small$g$};
                    \fill ($(newE)+(0,-\p/2)$)    circle (0.02in) node[right] {\small$g_1$};
                    \fill ($(newG)+(-\p/2,\p/2)$) circle (0.02in) node[right] {\small$g_2$};

                    \node[blue]    at (0.5,1.25) {\small$\eta_1'$};
                    \node[cyan]    at (3,1.4)    {\small$\eta_2'$};
                    \node[purple]  at (2.4,0.5)  {\small$\eta_3'$};
                    \node[teal]    at (1.3,-0.4) {\small$\eta_4'$};

                    \node[magenta]             at (3.2,0.7){\small$\zeta_{j1}$};
                    \node[magenta, rotate=180] at (2.8,0.7) {\small$\rightsquigarrow$};

                    \node[pink, yshift=0.35cm]                          at (newI) {\small$\eta$};
                    \node[pink, xshift=0.3cm, yshift=0.35cm, rotate=10] at (newI) {\small$\rightsquigarrow$};

                    \node[darkGreen, opacity=0.7] at (4.4,2.1) {\small Good};
                    \node at (-1,2.1) {\small$\mc{C}(G_{n-2})$};
                \end{scope}
            \end{tikzpicture}
        \end{center}
        \begin{minipage}{0.95\textwidth}
            \vspace{0.1in}
            \caption{In the general case, an additional extension is needed to obtain a `central' vertex $g\in\Omega\cap\Delta_j$. This\vspace{-0.02in} defines paths $\eta:1\to g$ and $\zeta_{j1}:g_{j-1}\to g$, which we can use to cut up $\eta_1'\cdots\eta_n'$ into two `shorter' loops at $g$; in the figure, the new loops are $\bar{\eta}\eta_1'\eta_2'\zeta_{j1}$ and $\bar{\zeta_{j1}}\eta_3'\eta_4'\eta$, projected appropriately to $\Gamma$.}
            \label{fig:proof_n}
            \vspace{-0.1in}
        \end{minipage}
    \end{figure}

    Take $g\in\Delta_j\cap\Omega$ for some fixed $1<j<n$, which we will use as our now starting point, at which we apply induction on certain collections of paths of size $j$ and $n-j+1$; the base case $n=2$ was handled previously. To construct these new paths, let $\zeta_{j1}:g_{j-1}\to g$ be the reduced path in $\Delta_j$ such that $\ell(\zeta_{j1})$ is a prefix of $\ell(\eta_j')$, and let $\eta:1\to g$ be a reduced path in $\Omega$ such that $\ell(\eta)$ is a prefix of $\ell(\eta_1')$; see Figure \ref{fig:proof_n}.

    Since $\ell(\gamma_j')=\ell(\eta_j')$, we can use Lemma \ref{lem:create_paths} to project $\zeta_{j1}$ down to a path $\delta_{j1}'$ in $\Gamma$ such that $\ell(\delta_{j1}')=\ell(\zeta_{j1})$ and $\alpha(\delta_{j1}')=\alpha(\gamma_j')$. Similarly, applying Lemma \ref{lem:create_paths} twice, we can project $\eta$ down to paths $\beta_1$ and $\beta_n$ in $\Gamma$ such that $\ell(\beta_1)=\ell(\beta_n)=\ell(\eta)$, $\alpha(\beta_1)=\alpha(\gamma_1')$, and $\alpha(\beta_n)=\omega(\gamma_n')$. The paths $\delta_{j1}'$, $\beta_1$, and $\beta_n$ in $\Gamma$ can now be used to cut our original cycle $\gamma_1'\cdots\gamma_n'$ at $1$ into two cycles at (images of) $g$, as follows; see Figure \ref{fig:proof_n}.

    \begin{itemize}
        \item For each $i\not\in\l\{1,j,n\r\}$, let $\delta_i'\coloneqq\gamma_i'$.
            \vspace{-0.07in}
        \item Let $\delta_1'$ be the (unique) reduced path homotopic to $\bar{\beta_1}\gamma_1'$, so $\ell(\delta_1')=\ell(\bar{\eta})w_1$.
            \vspace{-0.07in}
        \item Similarly define $\delta_{j2}'$ (resp. $\delta_n'$) by reducing $\bar{\delta_{j1}'}\gamma_j'$ (resp. $\gamma_n'\beta_n$), so $\ell(\delta_{j2}')=\ell(\bar{\zeta_{j1}})w_j$ and $\ell(\delta_n')=w_n\ell(\eta)$.
    \end{itemize}
    The collections $\{\delta_1',\dots,\delta_{j-1}',\delta_{j1}'\}$ and $\{\delta_{j2}',\delta_{j+1}'\dots,\delta_n'\}$ of paths lie in $\Gamma$, but since we have
    \vspace{-0.05in}
    \makeatletter
        \newcommand*\bigcdot{\mathpalette\bigcdot@{.5}}
        \newcommand*\bigcdot@[2]{\mathbin{\vcenter{\hbox{\scalebox{#2}{$\m@th#1\bullet$}}}}}
    \makeatother
    \begin{equation*}
        [(\ell(\bar{\eta})w_1)\bigcdot w_2\bigcdot\cdots\bigcdot w_{j-1}\bigcdot\ell(\zeta_{j1})]_{G_{n-2}}=1=[(\ell(\bar{\zeta_{j1}})w_j)\bigcdot w_{j+1}\bigcdot\cdots\bigcdot w_{n-1}\bigcdot(w_n\ell(\eta))]_{G_{n-2}},
        \vspace{-0.05in}
    \end{equation*}
    they can be lifted to collections $\{\zeta_1,\dots,\zeta_{j-1},\zeta_{j1}\}$ and $\{\zeta_{j2},\zeta_{j+1}\dots,\zeta_n\}$ of paths in $\mc{C}(G_{i-1})$ and $\mc{C}(G_{n-i})$, so that, placing them tip-to-tail, they form two cycles at the image of $g$ in $G_{i-1}$ and $G_{n-i}$, respectively; all this works since $1\leq i-1,n-i\leq n-2$. By induction, we obtain paths $\delta_1,\dots,\delta_{j-1},\delta_{j1}$ and $\delta_{j2},\delta_{j+1}\dots,\delta_n$ in $\Gamma$ with the same initial and terminal points with $\delta_1',\dots$ etc, such that $[\delta_1\cdots\delta_{j-1}\delta_{j1}]_F=1=[\delta_{j2}\delta_{j+1}\cdots\delta_n]_F$.

    Finally, we can let $\gamma_i\coloneqq\delta_i$ for all $i\not\in\l\{1,j,n\r\}$, let $\gamma_1\coloneqq\beta_1\delta_1$, let $\gamma_j\coloneqq\delta_{j1}\delta_{j2}$, and let $\gamma_n\coloneqq\delta_n\bar{\beta_n}$. It can be easily checked that the initial and terminal points of $\gamma_i$ and $\gamma_i'$ coincide, hence preserving (1), and now
    \vspace{-0.05in}
    \begin{equation*}
        [\gamma_1\cdots\gamma_n]_F=[\beta_1\delta_1\cdots\delta_{j-1}\delta_{j1}\delta_{j2}\delta_{j+1}\cdots\delta_n\bar{\beta_n}]_F=[\beta_1\bar{\beta_n}]_F=[\eta]_F[\bar{\eta}]_F=1.
        \vspace{-0.05in}
    \end{equation*}
    Thus the paths $\gamma_1,\dots,\gamma_n$ satisfy condition (2) too, and this finishes the proof of Theorem \ref{mthm}.\qed

    \printbibliography

    {\footnotesize
        \vspace{0.2in}
        \textsc{Department of Mathematics and Statistics, McGill University, 805 Sherbrooke Street West, Montreal, QC, H3A 0B9, Canada}

        \textit{Email address: \tt{zhaoshen.zhai@mail.mcgill.ca}}
    }
\end{document}